\documentclass{article}

\usepackage[english]{babel}

\usepackage[letterpaper,top=2cm,bottom=2cm,left=3cm,right=3cm,marginparwidth=1.75cm]{geometry}

\usepackage{amsmath}
\usepackage{graphicx}
\usepackage[colorlinks=true, allcolors=blue]{hyperref}
\newcommand{\norm}[1]{\left\lVert#1\right\rVert}
\usepackage{blindtext}
\usepackage{caption}
\usepackage{subcaption}
\usepackage{pifont}
\usepackage{amsfonts}
\usepackage{amssymb}
\usepackage{algorithm}
\usepackage{algpseudocode}
\usepackage{blkarray, bigstrut}
\usepackage{amsthm}
\newtheorem{theorem}{Theorem}
\usepackage{subcaption}
\newtheorem{lemma}[theorem]{Lemma}
\usepackage{comment}
\usepackage{mathtools}
\usepackage{caption}
\usepackage{comment}

\DeclareMathOperator{\SPAN}{span}

\title{Efficient Function Approximation in \\Enriched Approximation Spaces}
\author{Astrid Herremans and Daan Huybrechs}
\date{}

\begin{document}
\maketitle

\begin{abstract}
An enriched approximation space is the span of a conventional basis with a few extra functions included, for example to capture known features of the solution to a computational problem. Adding functions to a basis makes it overcomplete and, consequently, the corresponding discretized approximation problem may require solving an ill-conditioned system. Recent research indicates that these systems can still provide highly accurate numerical approximations under reasonable conditions. In this paper we propose an efficient algorithm to compute such approximations. It is based on the AZ algorithm for overcomplete sets and frames, which simplifies in the case of an enriched basis. In addition, analysis of the original AZ algorithm and of the proposed variant gives constructive insights on how to achieve optimal and stable discretizations using enriched bases. We apply the algorithm to examples of enriched approximation spaces in literature, including a few non-standard approximation problems and an enriched spectral method for a 2D boundary value problem, and show that the simplified AZ algorithm is indeed stable, accurate and efficient.
\end{abstract}

\section{Introduction}
For many computational problems arising in science and engineering, it is a difficult task to incorporate knowledge on the behaviour of the solution into a robust approximation method. An expert practioner can often readily identify certain functions capturing dominant characteristics of the solution. On the other hand, approximating with such non-standard basis functions turns out to be challenging, as it generally leads to ill-conditioned linear systems. Recent work \cite{adcock2019frames,adcock2020fna2} based on frames theory however indicates that highly accurate solutions can still be found if both the approximation set and the discretization are, in some sense, sufficiently rich. The first condition is associated with the need for bounded coefficient vectors, i.e., the coefficients multiplying the basis functions should not grow too large. The latter results in a shift towards least squares fitting instead of solving square systems. When these conditions are met, it can be shown that regularization mitigates the ill-conditioning concerns. These results are a strong motivation for the above-mentioned \textit{expert-driven} approximation strategy.

In this paper, we restrict our focus to enriched approximation sets consisting of a conventional basis augmented with a few extra functions, which capture certain known features of a function to be approximated. Settings in which these approximation sets may arise are plentiful. An important setting is when the solution exhibits singular behaviour, see for example generalized/extended finite element methods \cite{fix1973use,fries2010extended} and enriched spectral methods \cite{chen2018enriched,gopal2019solving}. Another context is when the solution is periodized, see for example pseudo-spectral methods \cite{roache1978pseudo}, or exhibits known oscillatory behaviour~\cite{gibbs2021high}. 

Our aim is to find accurate approximations in such enriched sets via efficient least squares fitting, for which we propose to use (a variant of) the AZ algorithm~\cite{coppe2020az}. The AZ algorithm originated with an efficient method to compute Fourier extension approximations~\cite{matthysen2016fastfe}, where the focus laid on manipulating the singular value profile of the system matrix. In a much broader sense, the algorithm can be interpreted as a strategy to reduce the dimensionality of the least squares problem using an efficient solver for a partial problem. In this paper, we make this general interpretation of the AZ algorithm rigorous and propose a constructive simplification of the algorithm in the case of enriched bases. 

In section~\ref{sec:2}, we review the numerical aspects of computing approximations in overcomplete sets as well as the AZ algorithm and its main properties. Also, two novel interpretations of the algorithm are given. A simplification of the AZ algorithm for enriched bases is then proposed in section~\ref{sec:3}. Analysis of the algorithm furthermore leads to constructive insights on how to optimally discretize these non-standard approximation sets. In section~\ref{sec:4}, the algorithm is used to compute approximations in two common examples of enriched spaces. Ultimately, it is shown in section~\ref{sec:5} how the algorithm can be used to efficiently compute enriched solutions in the context of boundary value problems. To this end, an existing enriched spectral method is interpreted as an AZ algorithm and thereafter adapted to remove certain smoothness constraints on the approximation set. The code to reproduce all numerical experiments can be found in~\cite{azcode}.

\section{The AZ algorithm for overcomplete sets} \label{sec:2}

\subsection{Discrete least squares approximation in overcomplete sets} \label{sec:fna2bound}

Consider the problem of approximating a function $f$ in a finite approximation set $\Phi_N = \{\phi_n\}_{n=1}^N$ on a domain $\Omega \subset \mathbb{R}^d$ based on discrete data. These \textit{data points} consist of $M$ samples $\{f(x_m)\}_{m=1}^M$ or, more generally, of functionals $\boldsymbol{\xi}(f) = \{\xi_{m}(f)\}_{m=1}^M$. The discrete best approximation can then be computed by solving a (rectangular) linear system 
\begin{equation}
    A\mathbf{x} = \mathbf{b}
    \label{eq:system}
\end{equation}
with $A \in \mathbb{C}^{M \times N}$, $A_{m,n} = \xi_{m}(\phi_n)$ and $b_m = \xi_{m}(f)$.

Recent work~\cite{adcock2019frames,adcock2020fna2} shows that the system matrix of~\eqref{eq:system} is generally highly ill-conditioned when approximating in an overcomplete set $\Phi$, yet that accurate results can often still be obtained. The ill-conditioning indicates that the error on the coefficients $\mathbf{x}$ can be arbitrarily large, which is a natural consequence of the (near-)redundancy in the approximation set. However, for function approximation only the residual is of interest, which can still be small when using effective regularization. In \cite[Lemma 3.3]{coppe2020az}, a bound is given on the residual when using Truncated Singular Value Decomposition (TSVD) regularization where the singular values of A below a threshold $\epsilon$ are truncated:
\begin{equation}
    \norm{\mathbf{b} - A\mathbf{x}}_2 \leq \inf_{\mathbf{c} \in \mathbb{C}^N} \{ \kern2pt \norm{\mathbf{b} - A\mathbf{c}}_2 + \epsilon \norm{\mathbf{c}}_2 \kern2pt \}.
    \label{eq:tsvdbound}
\end{equation}
The bound shows that the regularized solver strikes a balance between the residual and the coefficient norm of the solution (multiplied by the regularization threshold). It indicates that the residual can be as small as $\mathcal{O}(\epsilon)$, if there exists a solution in the approximation set with a small residual $\norm{\mathbf{b} - A\mathbf{c}}_2$ as well as a bounded coefficient vector $\norm{\mathbf{c}}_2$. It is therefore important to inspect the norm of the coefficient vector when approximating in an overcomplete set, since the two terms in~\eqref{eq:tsvdbound} are balanced. A larger coefficient norm therefore also corresponds to a larger residual and, hence, less accuracy.

On the other hand, the function $f$ is generally an element of a Hilbert space H endowed with a norm $\norm{\cdot}_H$ such that one is actually interested in an accurate approximation with respect to this H-norm. The approximation error can then be bounded by~\cite[Thm. 3.1 and Prop. 3.10]{adcock2020fna2}
\begin{equation}
    \norm{f-P^{\epsilon}_{M,N}f}_H \leq \inf_{\mathbf{c} \in \mathbb{C}^N} \left\{ \; \norm{f - \sum_{n=1}^Nc_n\phi_n}_H \; + \;  \frac{1}{\sqrt{A_{M,N}'}} \left( \kern2pt \norm{\mathbf{b} - A \mathbf{c}}_2 \; + \; \epsilon \norm{\mathbf{c}}_2 \kern2pt \right) \; \right\},
\end{equation}
where $P^\epsilon_{M,N}f$ denotes the approximation whose coefficients are the solution of~\eqref{eq:system} again using TSVD regularization. Besides the discrete residual and the norm of the coefficient vector, the error bound also contains the residual in the H-norm and a scaling factor $1/\sqrt{A_{M,N}'}$. The constant $A_{M,N}'$ measures the equivalence between the discrete norm based on the sampling functionals and the H-norm, and is defined by
\begin{equation}
    A_{M,N}' \kern1pt = \inf_{\substack{f \kern2pt \in \kern2pt \SPAN(\Phi_N) \\ \norm{f}_H=1}} \norm{\boldsymbol{\xi}(f)}_2^2.
\end{equation}
This constant is independent of the representation $\Phi_N$ and solely depends on the richness of the sampling functionals $\boldsymbol{\xi}(\cdot)$ with respect to functions in $\SPAN(\Phi_N) \subset H$. When it is small, the scaling factor $1/\sqrt{A_{M,N}'}$ grows large such that a small discrete residual does not ensure high accuracy in the continuous setting. For $A_{M,N}'$ to be bounded from below one generally needs either a judicious choice of sampling functions when $M=N$, but simpler than that is to oversample such that $M > N$ resulting in rectangular linear systems \cite{adcock2020fna2, cohen2017, GROCHENIG2020105455}. 

\subsection{AZ algorithm} \label{sec:generalaz}
The AZ algorithm~\cite[Algorithm 2.1]{coppe2020az}, recited as Algorithm \ref{alg:generalaz}, aims at solving the linear system~\eqref{eq:system} efficiently by constructing a new least squares system $A-AZ^*A$ which is low-rank. It does so by using a matrix $Z^*$ which functions as an incomplete generalized inverse. The efficiency of the AZ algorithm therefore hinges on a good choice of the matrix $Z$. In~\cite{coppe2020az}, it is shown that the construction of such a matrix is often linked to a dual frame, as illustrated for several examples including extension frames and weighted linear combinations of bases. The AZ algorithm is already used for the efficient computation of Fourier~\cite{matthysen2016fastfe}, spline~\cite{coppe2022efficient} and wavelet~\cite{coppe2020efficient} extensions.

\begin{algorithm}
\caption{The AZ algorithm}\label{alg:generalaz}
\textbf{Input:} $A,Z \in \mathbb{C}^{M\times N}, b \in \mathbb{C}^M$ \newline
\textbf{Output:} $\mathbf{x} \in \mathbb{C}^N$ such that $A\mathbf{x} \approx \mathbf{b}$ in least squares sense
\begin{algorithmic}[1]
\State Solve $(A-AZ^*A)\mathbf{x}_1 = (I-AZ^*)\mathbf{b}$
\State $\mathbf{x}_2 \gets Z^*(\mathbf{b}-A\mathbf{x}_1)$
\State $\mathbf{x} \gets \mathbf{x}_1 + \mathbf{x}_2$
\end{algorithmic}
\end{algorithm}

We first restate the computational complexity and error analysis of the algorithm. Thereafter two novel interpretations of the algorithm are given: one from an algebraic point of view and one from an analytic point of view. The first aids in interpreting the matrix $Z^*$ as a partial solver, i.e.\ for many sampled functions multiplication with $Z^*$ returns the coefficients of an accurate approximant. The latter shows that the AZ algorithm essentially performs a change of basis, i.e.\ it uses the partial solver to switch to a lower-dimensional fitting problem.

\subsubsection{Properties of the AZ algorithm}
The AZ algorithm aims at efficiently solving~\eqref{eq:system} by creating a new least squares problem with system matrix $A-AZ^*A$. The newly obtained system matrix should be low-rank in order to have a reduced computational cost compared to solving the original system $A\mathbf{x}\approx \mathbf{b}$. Using a randomized TSVD solver, one can exploit this low rank resulting in a computational cost
\begin{equation}
    \mathcal{O}(RT_{mult} + MR^2) \text{ flops,}
\end{equation}
where $R$ equals the (numerical) rank of $A-AZ^*A$ and $T_{mult}$ equals the cost of a matrix-vector multiplication with $A-AZ^*A$. This also reveals a new condition on $A$ and $Z^*$ needed to obtain an efficient AZ algorithm: both matrices should have efficient matrix-vector multiplications.

In essence, the discrete approximation problem is solved approximately, since the AZ algorithm only uses expensive least squares fitting for a lower-dimensional subproblem. The accuracy could therefore decrease compared to using least squares fitting for the complete problem. A key element in analysing the error of the AZ algorithm is the fact that the final residual is equal to the residual of the first AZ equation~\cite[Lemma 2.1]{coppe2020az}, i.e.\ the second AZ equation does not introduce an error. As explained in~\S\ref{sec:fna2bound}, the achievable accuracy of least squares approximations in overcomplete sets depends not only on the size of the residual but also on the norm of the coefficient vector. In~\cite[Lemma 2.2]{coppe2020az} a bound is given for the growth of both of these quantities when using the AZ algorithm to solve~\eqref{eq:system}.
\begin{lemma}[{\cite[Lemma 2.2]{coppe2020az}}] \label{lm:error}
    Let $A\in \mathbb{C}^{M \times N}$, $\mathbf{b} \in \mathbb{C}^M$, and suppose there exists $\Tilde{\mathbf{x}} \in \mathbb{C}^N$ such that 
    \[
    \norm{\mathbf{b} - A\Tilde{\mathbf{x}}}_2 \leq \tau, \quad \norm{\Tilde{\mathbf{x}}}_2 \leq C,
    \]
    for $\tau, C > 0$. Then there exists a solution $\hat{\mathbf{x}}_1$ to step 1 of the AZ algorithm such that the residual of the computed vector $\hat{\mathbf{x}} = \hat{\mathbf{x}}_1 + \hat{\mathbf{x}}_2$ satisfies,
    \[
        \norm{\mathbf{b} - A\hat{\mathbf{x}}}_2 \leq \norm{I-AZ^*}_2\tau, \quad \norm{\hat{\mathbf{x}}}_2 \leq C + \norm{Z^*}_2 \tau.
    \]
\end{lemma} 
Since $\norm{I-AZ^*}_2 \leq 1 + \norm{A}_2\norm{Z^*}_2$, accurate solutions are guaranteed when $\norm{A}_2$ and $\norm{Z}_2$ are sufficiently bounded.

\subsubsection{Algebraic interpretation: annihilator \boldmath{$(I-AZ^*)$}}
Recall that a true generalized inverse $A^g$ of a matrix $A$ satisfies 
\[
AA^gA=A.
\]
It means that $\mathbf{x}=A^g\mathbf{b}$ solves the linear system~\eqref{eq:system} whenever the right hand side has the form $\mathbf{b}=A\mathbf{c}$ for some vector $\mathbf{c}$, i.e., it is in the column space of $A$. Indeed, in that case $\mathbf{x}=A^g\mathbf{b}$ leads to
\[
 A\mathbf{x} = A A^g\mathbf{b} = A A^g A \mathbf{c} = A\mathbf{c} = \mathbf{b}.
\]

Now assume that $Z^*$ is an incomplete generalized inverse in the sense that
\[
A - A Z^* A
\]
has low rank, instead of being zero. That means that $\mathbf{x}=Z^*\mathbf{b}$ is a solver for~\eqref{eq:system} on a large part of the column space of $A$, though possibly not all of it.

One can split the right hand side of~\eqref{eq:system} as $\mathbf{b} = \mathbf{b}_1 + \mathbf{b}_2$, in which $\mathbf{b}_2$ is a suitable right hand side for $Z^*$, i.e., $Ax=\mathbf{b}_2$ is solved by $Z^*\mathbf{b}_2$. This is equivalent to $AZ^*\mathbf{b}_2 = \mathbf{b}_2$ or
\begin{equation}\label{eq:annihilator}
 (I - AZ^*) \mathbf{b}_2 = 0.
\end{equation}
One can think of $I - AZ^*$ as an annihilator for most of the column space of $A$.

An analogous division for the unknown of~\eqref{eq:system} results in $\mathbf{x} = \mathbf{x}_1 + \mathbf{x}_2$ where $\mathbf{x}_2 = Z^* \mathbf{b}_2$. From~\eqref{eq:annihilator} it follows that
\[
 (I - AZ^*) A \mathbf{x}_2 = (I - AZ^*) \mathbf{b}_2 = 0.
\]
By multiplying both sides of~\eqref{eq:system} by $I - AZ^*$, one therefore arrives at an equation for $\mathbf{x}_1$:
\[
 (I-AZ^*) A \mathbf{x} =  (I-AZ^*) A (\mathbf{x}_1 + \mathbf{x}_2) =  (I-AZ^*) A \mathbf{x}_1 =  (I-AZ^*) \mathbf{b}.
\]
This is exactly the \emph{first AZ equation}
\[
 (A - A Z^*A) \mathbf{x}_1 = b - AZ^*\mathbf{b}.
\]
By the assumption on $Z^*$, this is a linear system with low rank. One recovers $\mathbf{x}_2=Z^* \mathbf{b}_2$ using $\mathbf{b}_2 = \mathbf{b}- \mathbf{b}_1$, which with $\mathbf{b}_1 = A \mathbf{x}_1$ leads to the \emph{second AZ equation}
\[
 \mathbf{x}_2 = Z^* (\mathbf{b} - A \mathbf{x}_1).
\]
Finally,
\[
 \mathbf{x} = \mathbf{x}_1 +  \mathbf{x}_2.
\]

This algebraic derivation of the algorithm facilitates the intepretation of the AZ algorithm in the context of approximation theory. It shows that the matrix $Z^*$ functions as a \textit{partial solver}: for many sampled functions, one obtains the coefficients of an accurate approximation in the set $\Phi_N$ by multiplication with $Z^*$.

\subsubsection{Analytic interpretation: change of basis \boldmath{$(I-Z^*A$)}} \label{sec:changeofbasis}

In this section the notation of~\cite{adcock2019frames} is used. Consider the approximation problem as outlined in \S\ref{sec:fna2bound}. The AZ algorithm does not directly compute the discrete best approximation of $f$ in the set $\Phi_N$, but first applies a partial solver $Z^*$ to its sampled data $\mathbf{b}$
\[
f = \mathcal{T}_NZ^*\mathbf{b} + \Tilde{f},
\]
where $\mathcal{T}_N$ denotes the \textit{synthesis operator} associated to $\Phi_N$, which can be written as a quasi-matrix of size $\infty \times N$
\[
 \mathcal{T}_{N} = \begin{bmatrix} \phi_1 & \phi_2 & \dots & \phi_{N} \end{bmatrix}.
\]
The problem then shifts to approximating the remainder function $\Tilde{f}$. 

Assume for a moment that $f$ is in $\SPAN(\Phi_N)$, such that $f = \mathcal{T}_N\mathbf{c}$ and $\mathbf{b} = A\mathbf{c}$ for some vector $\mathbf{c}$ without approximation error. It then follows that $\Tilde{f}$ can be rewritten as
\[
\Tilde{f} = f - \mathcal{T}_NZ^*\mathbf{b} = \mathcal{T}_N\mathbf{c} - \mathcal{T}_NZ^*A\mathbf{c} = \mathcal{T}_N(I-Z^*A) \mathbf{c}.
\]
In general, it is therefore natural to approximate the remainder function in a new approximation set $\Tilde{\Phi}_N$ with synthesis operator $\Tilde{\mathcal{T}}_N$
\begin{equation}
    \Tilde{f} \approx \Tilde{\mathcal{T}}_N\mathbf{x}_1 \quad \text{ with }\quad \Tilde{\mathcal{T}_N} = \mathcal{T}_N(I-Z^*A).
    \label{eq:phitilde}
\end{equation}
This problem is solved in the first AZ equation: one computes the discrete best approximation of $\Tilde{f}$ in $\Tilde{\Phi}_N$. This becomes clear when writing the system matrix $A$ as
\[
A = \mathcal{M}_M \mathcal{T}_{N}
\]
using the \textit{sampling operator} $\mathcal{M}_M: f \mapsto \boldsymbol{\xi}(f)$ and noting that the system matrix of the first AZ equation can be rewritten accordingly:
\[
 A - AZ^*A = A(I-Z^*A) = \mathcal{M}_M \mathcal{T}_{N}(I-Z^*A) = \mathcal{M}_M \Tilde{\mathcal{T}}_{N}.
\]
Furthermore, the right-hand side of the first AZ equation indeed contains the data of the remainder function $\Tilde{f}$
\[
(I-AZ^*)\mathbf{b} = \mathcal{M}_M f - \mathcal{M}_M \mathcal{T}_N Z^*\mathbf{b} = \mathcal{M}_M \Tilde{f}.
\]
The second AZ equation immediately follows from regrouping the coefficients of the approximant
\begin{align*}
    f &\approx \mathcal{T}_NZ^*\mathbf{b} + \Tilde{\mathcal{T}}_{N}\mathbf{x}_1 \\ &= \mathcal{T}_NZ^*\mathbf{b} + \mathcal{T}_{N}(I-Z^*A)\mathbf{x}_1 \\ &= \mathcal{T}_{N}\mathbf{x}_1 + \mathcal{T}_{N}(Z^*\mathbf{b} - Z^*A\mathbf{x}_1)
    \\ &\eqqcolon \mathcal{T}_{N}\mathbf{x}_1 + \mathcal{T}_{N}\mathbf{x}_2 \\
    &= \mathcal{T}_N \mathbf{x}
\end{align*}
from which we see that
\[
    \mathbf{x}_2 = Z^*(\mathbf{b} - A\mathbf{x}_1).
\]

It is interesting to inspect the newly obtained basis $\Tilde{\Phi}_N$. From~\eqref{eq:phitilde}, it follows that $\Tilde{\Phi}_N = \{\Tilde{\phi}_i\}_{i=1}^N$ is defined by
\[
\Tilde{\phi}_i = \phi_i - \sum_{i=1}^N c_i \phi_i \quad \text{ where } \mathbf{c} = Z^* \mathcal{M}_M\phi_i.
\]
The approximation set thus consists of the original basis functions $\{\phi_i\}^N_{i=1}$ minus their approximation in $\Phi$ using the partial solver $Z^*$. The system $A-AZ^*A$ being low-rank translates into the newly obtained basis functions $\Tilde{\Phi}$ spanning a substantially smaller (sampled) space than the original approximation set $\Phi$. We make these observations concrete in some examples further on.

\section{The AZ algorithm for enriched bases} \label{sec:3}

A set of $N$ conventional basis functions $\{\varphi_n\}_{n=1}^N$ enriched with $K$ extra functions $\{\psi_k\}_{k=1}^K$ leads to a finite approximation set
\begin{equation}
    \Phi_{N+K} = \{ \phi_i \}_{i=1}^{N+K} = \{\varphi_n\}_{n=1}^N \cup \{\psi_k\}_{k=1}^K,
    \label{eq:enrichedbasis}
\end{equation}
which we term an \textit{enriched basis}. A function $f$ can be approximated in this enriched basis by computing the discrete best approximation~\eqref{eq:system} using data points such as samples $f(x_m)$ or, more generally, functionals $\xi_m(f)$. 

It is natural to assume that an efficient solver already exists to compute approximations in the conventional basis $\{\varphi_n\}_{n=1}^N$ using $M_N$ data points, which often exhibit structure that enables efficient operations. In general, one also adds $M_K$ data points to sufficiently sample the behaviour of the additional functions. These points usually do not have similar structure, but their number is small. The total number of data points equals $M_N + M_K$, resulting in a system matrix
\begin{equation}
    A = \begin{blockarray}{ccc}
 & N & K \\
\begin{block}{c[cc]}
M_N & A_{11} & A_{12} \bigstrut[t] \\
M_K & A_{21} & A_{22} \bigstrut[b]\\
\end{block}
\end{blockarray}
\label{eq:A}
\end{equation}
with $A \in \mathbb{C}^{(M_N + M_K) \times (N + K)}$. 

From \S\ref{sec:generalaz} it follows that the AZ algorithm is an efficient algorithm to solve \eqref{eq:system}, if $Z^*$ is chosen as an incomplete generalized inverse of $A$ such that $A-AZ^*A$ is low-rank. This is equivalent to $Z^*$ being a solver for \eqref{eq:system} on a large part of the column space of A. One way to achieve this is by constructing it using the existing solver for the basis $\{\varphi_n\}_{n=1}^N$. Denoting the latter by $Z_{11}^*\in \mathbb{C}^{N \times M_N}$, one can simply construct a matrix $Z^*$ as
\begin{equation}
    Z^* = \begin{blockarray}{ccc}
     & M_N & M_K \\
    \begin{block}{c[cc]}
    N & Z^*_{11} & 0 \bigstrut[t] \\
    K & 0 & 0 \bigstrut[b]\\
    \end{block}
    \end{blockarray}.
\label{eq:Z}
\end{equation}
Multiplication by $Z^*$ then solves \eqref{eq:system} for each function in $\SPAN(\{\varphi_n\}_{n=1}^N)$. This is a large part of the column space of $A$, if one assumes that the number of extra functions $K$ is small. Based on the properties of the existing solver $Z_{11}^*$, one obtains different simplifications of the AZ algorithm, where $A-AZ^*A$ is not only low-rank but also sparse.

\begin{theorem}[AZ algorithm for enriched bases] \label{thm}
    Consider Algorithm \ref{alg:generalaz} (the AZ algorithm \cite{coppe2020az}) to solve~\eqref{eq:system} with $A$ as defined by~\eqref{eq:A} and choosing $Z^*$ as defined by~\eqref{eq:Z}.
    \begin{enumerate}
        \item If $Z_{11}^*$ is the inverse of $A_{11}$, the system matrix of the first AZ equation equals
            \begin{equation}
                A - AZ^*A = \begin{bmatrix} 0 & 0 \\ 0 & A_{22} - A_{21}Z_{11}^*A_{12} \end{bmatrix}, \label{eq:thm1}
            \end{equation}
        where the non-zero block is of size $M_K \times K$.
        \item If $Z_{11}^*$ is a left inverse of $A_{11}$, the system matrix of the first AZ equation equals
            \begin{equation}
                A - AZ^*A = \begin{bmatrix} 0 & A_{12} - A_{11}Z_{11}^*A_{12} \\ 0 & A_{22} - A_{21}Z_{11}^*A_{12} \end{bmatrix}, \label{eq:thm2}
            \end{equation}
        where the non-zero block is of size $(M_N + M_K) \times K$.
        \item If $Z_{11}^*$ is a right inverse of $A_{11}$, the system matrix of the first AZ equation equals
        \begin{equation}
            A - AZ^*A = \begin{bmatrix} 0 & 0 \\ A_{21} - A_{21}Z_{11}^*A_{11} & A_{22} - A_{21}Z_{11}^*A_{12} \end{bmatrix},
            \label{eq:thm3}
        \end{equation}
        where the nonzero block is of size $M_K \times (N + K)$.
    \end{enumerate}
\end{theorem}
\begin{proof}
    Computing $A-AZ^*A$ using \eqref{eq:A} and \eqref{eq:Z} results in
    \begin{align*}
        A - AZ^*A &= \begin{bmatrix} A_{11} & A_{12} \\ A_{21} & A_{22} \end{bmatrix} - \begin{bmatrix} A_{11} & A_{12} \\ A_{21} & A_{22} \end{bmatrix}\begin{bmatrix} Z_{11}^* & 0 \\ 0 & 0 \end{bmatrix} \begin{bmatrix} A_{11} & A_{12} \\ A_{21} & A_{22} \end{bmatrix} \\
        &= \begin{bmatrix} A_{11} & A_{12} \\ A_{21} & A_{22} \end{bmatrix} - \begin{bmatrix} A_{11}Z_{11}^* & 0 \\ A_{21}Z_{11}^* & 0 \end{bmatrix} \begin{bmatrix} A_{11} & A_{12} \\ A_{21} & A_{22} \end{bmatrix} \\ &= \begin{bmatrix} A_{11} - A_{11}Z_{11}^*A_{11} & A_{12} - A_{11}Z_{11}^*A_{12} \\ A_{21} - A_{21}Z_{11}^*A_{11} & A_{22} - A_{21}Z_{11}^*A_{12} \end{bmatrix}.
    \end{align*}
    This straightforwardly simplifies to \eqref{eq:thm1}, \eqref{eq:thm2} and \eqref{eq:thm3}.
\end{proof}

A few remarks are in order. Firstly, observe that when the matrix $A$ is square, the AZ algorithm with system matrix defined by~\eqref{eq:thm1} is equivalent to solving the linear system $A\mathbf{x} = \mathbf{b}$ using the Schur complement of $A$ relative to $A_{11}$ \cite{zhang2006schur}. In the current case, it is assumed that $A$ is rectangular and $A\mathbf{x} \approx \mathbf{b}$ can only be solved in a least squares sense. Secondly, note that no randomized solver is needed to exploit the structure of the matrices in Theorem~\ref{thm} due to their sparse block structure. The computational cost of the AZ algorithm is dominated by the cost of constructing $A-AZ^*A$ and solving the first AZ equation with the lower-dimensional system matrix. Thirdly, due to the choice of $Z^*$~\eqref{eq:Z}, the error of the AZ algorithm compared to solving~\eqref{eq:system} directly described in Lemma~\ref{lm:error}, can be simplified. The possible growth of the error now only depends on the norm of the subblocks $A_{11}$, $A_{21}$ and $Z_{11}$, which are independent of the extra functions $\{\psi_k\}_{k=1}^K$.

\begin{theorem}[Error of the AZ algorithm for enriched bases]
    \label{thm:enrichederror}
    If one defines $A$ by~\eqref{eq:A} and $Z^*$ by~\eqref{eq:Z} with $N,K,M_N,M_K > 0$,
    \[
        \norm{I-AZ^*}_2 \leq 1 + \norm{Z_{11}^*}_2(\norm{A_{11}}_2 + \norm{A_{21}}_2), \qquad \norm{Z^*}_2 = \norm{Z_{11}^*}_2.
    \]
    Therefore, under the same conditions as Lemma~\ref{lm:error}, there exists a solution $\hat{\mathbf{x}}_1$ to step 1 of the AZ algorithm for enriched bases such that the residual of the computed vector $\hat{\mathbf{x}} = \hat{\mathbf{x}}_1 + \hat{\mathbf{x}}_2$ satisfies,
    \begin{equation}
        \norm{\mathbf{b} - A\hat{\mathbf{x}}}_2 \leq \left( 1 + \norm{Z_{11}^*}_2(\norm{A_{11}}_2 + \norm{A_{21}}_2) \right)\kern2pt \tau, \qquad \norm{\hat{\mathbf{x}}}_2 \leq C + \norm{Z_{11}^*}_2 \tau.
        \label{eq:I_AZstar}
    \end{equation}
\end{theorem}
\begin{proof}
    One can rewrite $I-AZ^*$ as follows
    \[
        I-AZ^* = I - \begin{bmatrix}
            A_{11}Z_{11}^* & 0 \\ A_{21}Z_{11}^* & 0
        \end{bmatrix} = I - \begin{bmatrix}
            A_{11}Z_{11}^* & 0 \\ 0 & 0
        \end{bmatrix} - \begin{bmatrix}
            0 & 0 \\ A_{21}Z_{11}^* & 0
        \end{bmatrix}.
    \]
    From here, it follows that
    \[
     \norm{I-AZ^*}_2 \leq \norm{I}_2 + \norm{A_{11}Z_{11}^*}_2 + \norm{A_{21}Z_{11}^*}_2 \leq \norm{I}_2 + \norm{Z_{11}^*}_2 (\norm{A_{11}}_2 + \norm{A_{21}}_2).
    \]
    The expression for $\norm{Z^*}_2$ follows trivially from~\eqref{eq:Z}.
\end{proof}

A milder condition than those of Theorem~\ref{thm} is the assumption that an efficient AZ algorithm exists for approximation in $\{\varphi_n\}_{n=1}^N$. In this case, one has an incomplete generalized inverse $Z^*_{11}$ of $A_{11}$. Using the construction~\eqref{eq:Z} for the matrix $Z^*$, this leads to a generally non-sparse AZ algorithm for approximation in the enriched space.

\begin{lemma}
    Consider Algorithm \ref{alg:generalaz} (the AZ algorithm \cite{coppe2020az}) to solve~\eqref{eq:system} with $A$ as defined by~\eqref{eq:A} and choosing $Z^*$ as defined by~\eqref{eq:Z}. If $Z_{11}^*$ is an incomplete generalized inverse of $A_{11}$, in the sense that the rank of $A_{11}-A_{11}Z_{11}^*A_{11}$ equals $L < N$, the system matrix of the first AZ equation equals
    \begin{equation}
        A - AZ^*A = \begin{bmatrix} A_{11} - A_{11}Z_{11}^*A_{11} & A_{12} - A_{11}Z_{11}^*A_{12} \\ A_{21} - A_{21}Z_{11}^*A_{11} & A_{22} - A_{21}Z_{11}^*A_{12} \end{bmatrix}, \label{eq:thm4}
    \end{equation}
    with rank at most $L + M_K + K$.
\end{lemma} 
\begin{proof}
    The structure of~\eqref{eq:thm4} follows immediately from the proof of Theorem~\ref{thm}. The rank follows from viewing $A - AZ^*A$ as $[A_1 \; A_2]$ where the maximal rank of $A_1$ and $A_2$ equals $L+M_K$ and $K$ respectively.
\end{proof}
Our assumption that $K$ and $M_K$ are small implies that the rank of the system does not grow too much and the enriched AZ algorithm remains efficient in this case too.

\subsection{A constructive sampling strategy} \label{sec:changeofbasis2}
From an analytical point of view, the introduction of the matrix $Z$ can also be interpreted as a change of basis, see \S\ref{sec:changeofbasis}. To this end, note that the synthesis operator for an enriched basis $\Phi_{N+K}$~\eqref{eq:enrichedbasis} can be written as 
\[
 \mathcal{T}_{N+K} = \begin{bmatrix} \phi_1 & \phi_2 & \dots & \phi_{N+K} \end{bmatrix} = \begin{bmatrix} \varphi_1 & \dots & \varphi_N & \psi_1 & \dots & \psi_K \end{bmatrix} = \begin{bmatrix} \mathcal{T}_{N} & \mathcal{T}_{K} \end{bmatrix}.
\]
and the sampling operator can be written as $\mathcal{M}_M = \begin{bmatrix} \mathcal{M}_{M_N} \\ \mathcal{M}_{M_K} \end{bmatrix}$. Using~\eqref{eq:phitilde}, it follows that the basis of the first AZ equation equals 
\[
\Tilde{\mathcal{T}}_{N+K} = \mathcal{T}_{N+K}(I-Z^*A).
\]
Choosing $Z$ as defined by~\eqref{eq:Z}, one obtains
\begin{align*}
    \Tilde{\mathcal{T}}_{N+K} &= \mathcal{T}_{N+K} - \begin{bmatrix} \mathcal{T}_{N} & \mathcal{T}_{K} \end{bmatrix}
    \begin{bmatrix}
        Z_{11}^* & 0 \\ 0 & 0
    \end{bmatrix}     
    \begin{bmatrix}
        A_{11} & A_{12} \\ A_{21} & A_{22}
    \end{bmatrix} 
    \\ 
    &= \mathcal{T}_{N+K} - \mathcal{T}_{N} \begin{bmatrix}
        Z_{11}^*A_{11} & Z_{11}^*A_{12}
    \end{bmatrix}.
\end{align*}
When $Z_{11}^*$ is the inverse or a left inverse of $A_{11}$, this simplifies to 
\begin{align*}
    \Tilde{\mathcal{T}}_{N+K} &= \mathcal{T}_{N+K} -\mathcal{T}_N \begin{bmatrix}
        I & Z_{11}^* A_{12}
    \end{bmatrix} = \begin{bmatrix}
        0 & \mathcal{T}_K - \mathcal{T}_N Z_{11}^*A_{12}
    \end{bmatrix}.
\end{align*}
The new approximation set $\Tilde{\Phi}$ then only contains $K$ nonzero basis elements $\Tilde{\psi}_k$ defined by
\begin{equation}
    \Tilde{\psi}_k = \psi_k - \sum_{n=1}^N c_n \varphi_n \quad \text{ with } \mathbf{c} = Z_{11}^*\mathcal{M}_{M_N}\psi_k,
    \label{eq:changeofbasis}
\end{equation}
i.e. the set consists of the additional functions $\{\psi_k\}^K_{k=1}$ minus their approximation in the conventional basis $\{\varphi_n\}^N_{n=1}$ using the partial solver $Z_{11}^*$. 

The new basis functions are independent of the choice of the $M_K$ extra sampling functionals which define $\mathcal{M}_{M_K}$, due to the sparsity of $Z$~\eqref{eq:Z}. Inspecting the new approximation set $\Tilde{\Phi}$ therefore allows to gain insight in choosing the extra sampling functionals. Note that sufficiently rich data points are crucial to obtain accurate approximations, as explained in~\S\ref{sec:fna2bound}. 

Recent results on randomized sampling for $L^2$-approximations~\cite{cohen2017} state a constructive way to obtain a (near-)optimal sampling strategy, i.e.\ a strategy such that (log-)linear oversampling suffices to obtain an accurate least squares fit. To this end, one needs to randomly sample with respect to the inverse Christoffel function, which can be computed assuming an orthonormal basis for the approximation space is available. Throughout the following examples, we will use the inverse Chirstoffel function associated with $\SPAN(\Tilde{\Phi})$ as a tool to gain insight on how to distribute the $M_K$ extra sampling points, yet we will use deterministic samples for simplicity, as they give satisfactory results.

\subsection{Singular value profile}

The AZ algorithm originated with an efficient method to compute Fourier extension approximations which, in hindsight, is a special case in which one can choose $Z=A$~\cite{matthysen2016fastfe}. The method was motivated by the spectra of the matrices $A$ and $AA^*A$, and the relation between them. A number of extensions were explored in the PhD thesis~\cite[Chapter 5]{matthysen2018phd}  involving additional degrees of freedom and sampling points, much like in the current paper, but still using $Z=A$. Theorem~\ref{thm} above describes the rank of $A-AZ^*A$ in algebraic terms and more closely resembles a Schur complement approach than the original Fourier extension scheme. Yet, it remains instructive to examine and interpret the singular value profile of $A$~\eqref{eq:A} and $A-AZ^*A$ in the context of an enriched basis.

To this end, note that the subblock $A_{11}$ is assumed to be well-conditioned, as it consists of the evaluations of the conventional basis in the structured data points. The full matrix $A$ is obtained by adding $M_K$ rows and $K$ columns to $A_{11}$. This introduces (near-)redundancy and therefore causes singular values to approach zero, resulting in increasingly ill-conditioned matrices. Owing to the interlacing property of singular values of nested matrices, the spectrum of $A$ still consists of a large well-conditioned part and a smaller ill-conditioned part.

One formulation of the interlacing property is the following.
\begin{lemma}[\cite{thompson1972}, Theorem 1]\label{lem:interlacing}
Let $A$ be an $m \times n$ matrix with singular values $\alpha_1 \geq \alpha_2 \geq \ldots \geq \alpha_{\min\{m,n\}}$ and let $B$ be a $p \times q$ submatrix of $A$ with singular values $\beta_1 \geq \beta_2 \geq \ldots \geq \beta_{\min\{p,q\}}$. Then:
\[
\begin{array}{ll}
    \alpha_i \geq \beta_i, & i=1,2,\ldots,\min(p,q),\\
    \beta_i \geq \alpha_{i + (m-p) + (n-q)}, & i = 1,2,\ldots,\min(p+q-m,p+q-n).
\end{array}
\]
\end{lemma}
\begin{proof}
 The result is formulated in multiple places, the statement here is exactly that of \cite[Theorem 1]{thompson1972}. Its proof is based on the Cauchy interlacing theorem for Hermitian matrices (see, e.g., \cite[Theorem 8.1.7]{golub1996matrix}) applied to the Hermitian matrix $A^*A$. Note that the interlacing property is invoked for each additional row and each additional column of $A$ compared to $B$, separately, and that is why the index of $\alpha$ is $m-p+n-q$ larger than that of $\beta$ in the second inequality.
\end{proof}

It follows readily that the ill-conditioned part of $A$ has small dimension if both the number of extra functions $K$ and extra samples $M_K$ remain modest.
\begin{theorem}\label{th:interlacing}
Let the singular values of $A_{11}$, the top-left subblock of matrix $A$ given by~\eqref{eq:A}, be contained in the interval $[a,b]$ with $0 < a \leq b$. If $A$ is rectangular with $M_N + M_K \geq N+K$, then it has at most $K$ singular values smaller than $a$ and at most $K+M_K$ singular values larger than $b$.
\end{theorem}
\begin{proof}
The result is an application of  Lemma~\ref{lem:interlacing} with $B = A_{11}$ and, hence, $p=M_N$ and $q=N$. The matrix $A$ has $\min(m,n) = \min(M_N+M_K,N+K) = N+K$ singular values. It also follows from the oversampling condition that $\min(N-M_K,M_N-K) = N-M_K$. Thus, the interlacing inequalities of the previous lemma become
\[
\begin{array}{ll}
    \alpha_i \geq \beta_i, & i=1,2,\ldots,N,\\
    \beta_i \geq \alpha_{i + M_K + K}, & i = 1,2,\ldots,N-M_K.
\end{array}
\]
The conditions of the theorem stipulate that $\beta_1 \leq b$ and $\beta_N \geq a$. From the first inequality, we conclude that at least $N$ singular values of $A$ are bounded below by $a$, hence at most $K$ are possibly smaller. The second inequality shows that $N-M_K$ singular values of $A$ are bounded above by $b$, hence $K+M_K$ are possibly larger.
\end{proof}

Fig.\ \ref{fig:svd} (left) depicts the singular value profile of the system matrix $A$ introduced in example \S\ref{sec:WCextrapts}. Most of the singular values of $A$ are $O(1)$, since they interlace with those of the well-conditioned subblock $A_{11}$. The few extra singular values rapidly decay towards $0$. Similar spectra appear in matrices coming from sampling theory, such as subblocks of the Fourier DFT matrix \cite{barnett2022exponentially, dirckx2023computation}, and in the theory of time-frequency bandlimiting operators. In those contexts the ill-conditioned part is sometimes referred to as the \emph{plunge region} \cite[\S2.8]{daubechies1992tenlectures}. The singular value profile of the system matrix of the first AZ equation is shown in Fig.\ \ref{fig:svd} (right), i.e.\ it depicts the singular values of $A-AZ^*A$ where $Z$ is defined by~\eqref{eq:Z} with $Z_{11}^*$ a left inverse of $A_{11}$. It can be seen that applying the operator $I-AZ^*$ to $A$ largely corresponds to isolating its plunge region, as is the case for the original Fourier extension scheme~\cite{matthysen2016fastfe}.

\begin{figure}
\begin{subfigure}{.68\linewidth}
    \centering
    \includegraphics[width=\linewidth]{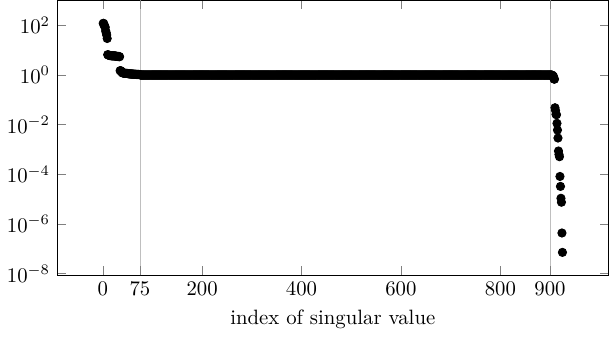}
\end{subfigure}\hfill
\begin{subfigure}{.29\linewidth}
    \centering
    \includegraphics[width=\linewidth]{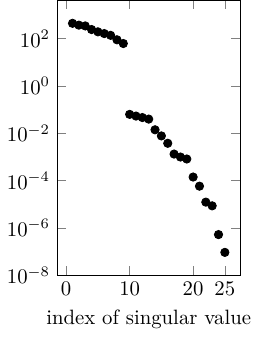}
\end{subfigure}
\caption{Singular value profile of the system matrix $A$ introduced in example~\S\ref{sec:WCextrapts}, having $N = 900$, $K = 25$, $M_N = 4N$ (here: using a cartesian product of Chebyshev nodes), $M_K = 2K$ (left) and the singular value profile of $A-AZ^*A$ (right). The vertical lines mark $K + M_K$ ($75$) and the total number of singular values minus $K$ ($900$). }
    \label{fig:svd}
\end{figure}

\section{Examples}\label{sec:4}

\subsection{Fourier series augmented with polynomials}
The Fourier basis suffers from the Gibbs phenomenon when it is used to approximate non-periodic functions. The problem can be reduced by augmenting the basis with a finite number of polynomials, an idea introduced by Krylov \cite{krylov} (see also \cite[Example 2]{adcock2019frames}). For approximation on $[0,1]$, this results in an approximation set
\[
\Phi_{N+K} = \{\varphi_n\}_{n=1}^N \cup \{ \psi_k \}_{k=1}^K = \{\kern1pt e^{2 \pi i n t} \kern1pt \}_{n=-\frac{N-1}{2}}^{\frac{N-1}{2}} \cup \{ \psi_k \}_{k=1}^K,
\]
where $N$ is assumed to be odd. For simplicity, we use Legendre polynomials $\psi_k$. Note that the constant polynomial $\psi_0$ can be excluded, as it is part of the Fourier basis. The $L^2$-convergence rate in this set is determined by the number of added polynomials $K$ \cite[Proposition 18]{adcock2019frames}. This effect can be explained by considering an approximation where the polynomials implicitly periodize the function and its $K-1$ derivatives, such that the Fourier coefficients of the new function decay more rapidly.

It is customary to compute the approximation by explicitly matching the derivatives of $f$ at the endpoints using a polynomial, subtracting that polynomial from $f$ and approximating the near-periodic remainder with an FFT \cite{eckhoff1998singularities,javed2016eulermaclaurin}. This technique is sometimes referred to as \emph{polynomial subtraction}. A least squares fit seems more expensive in comparison, yet it is both more stable and more accurate and, using AZ, can be implemented with similar complexity.

\subsubsection{Approximation using an oversampled equispaced grid}\label{sec:FLequi}

The discrete best approximation to a non-periodic function $f$ on an equispaced grid $\{t_m\}_{m=1}^M$ can be computed by solving
\begin{equation}
    \begin{bmatrix}
    A_{11} & A_{12} 
\end{bmatrix}
 \begin{bmatrix}
    \mathbf{x}_N \\ \mathbf{x}_K
\end{bmatrix}
 = \mathbf{b}
\label{eq:fourierlegendre}
\end{equation}
with $A_{11}\in \mathbb{C}^{M \times N}, A_{12}\in \mathbb{C}^{M \times K}$ and $\mathbf{b}_m = f(t_m)$. It is assumed that the system matrix is oversampled in the sense that $M > N + K$. This system can be solved efficiently using the AZ algorithm with a matrix Z \eqref{eq:Z} having $M_N = M$ and $M_K = 0$, where the matrix $Z_{11}^*$ is a left inverse of $A_{11}$. Note that a matrix-vector mulitplication with both $Z_{11}^*$ and $A_{11}$  can be computed using the FFT algorithm, requiring $\mathcal{O}(M \log{M})$ flops. Using the results of Theorem \ref{thm}, the first AZ equation simplifies to
\[
 (A_{12} - A_{11}Z^*_{11}A_{12}) \mathbf{x}_K \approx (I-A_{11}Z_{11}^*)\mathbf{b}
\]
which has a system matrix of size $M \times K$. The computational compexity of the AZ algorithm is then dominated by
\begin{itemize}
    \item construction of the system matrix: $\mathcal{O}(KM\log{M}) + \mathcal{O}(K^2M)$ flops,
    \item solving the least squares system: $\mathcal{O}(MK^2)$ flops.
\end{itemize}
In contrast, a regular least squares solver requires $\mathcal{O}(M(N+K)^2)$ flops. Assuming that $N \gg K$, the AZ algorithm is therefore much more efficient. Fig.\ \ref{fig:FLtiming} displays the timings of the AZ algorithm as well as of a regular solve of the least squares system \eqref{eq:fourierlegendre}, for $M = 2N$ and $K = 5$. Additionally, the results are compared to the randomized AZ algorithm introduced in \cite[\S5.1.2]{matthysen2018phd}, in which $Z = A$. The experiment was run on a contemporary laptop using an implementation in Julia.

\begin{figure}
\centering
\includegraphics[width=.65\textwidth]{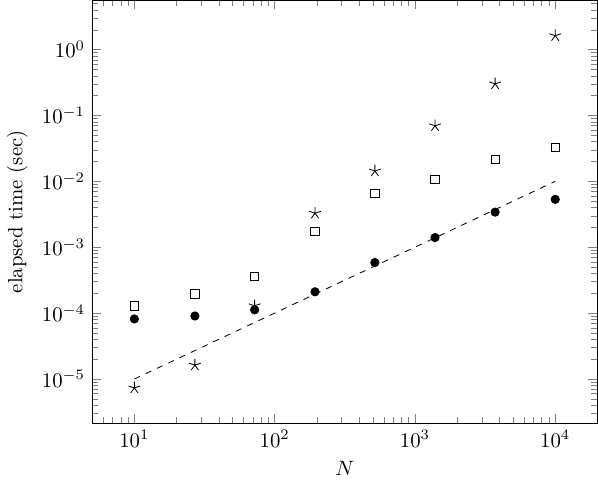}
\caption{Timings for computing the discrete best approximation~\eqref{eq:fourierlegendre} with an increasing number of Fourier basis functions and $K = 5$ Legendre polynomials. Stars: regular least squares solver (\texttt{ldiv} from Julia's LinearAlgebra module), squares: randomized AZ algorithm with $Z = A$ \cite[(5.22)]{matthysen2018phd}, dots: AZ algoritm for enriched bases. The dashed line marks $\mathcal{O}(N)$.}
\label{fig:FLtiming}
\end{figure}

\subsubsection{Approximation using an equispaced grid with extra points} \label{sec:FLextrapts}
The first AZ equation can also be interpreted as a new approximation problem after a change of basis, recall~\S\ref{sec:changeofbasis2}. In this case, the new basis consists of the Legendre polynomials minus their least squares Fourier series approximations. Fig.\ \ref{fig:FLbasis} shows the first element $\Tilde{\psi_1}$ of this new approximation set $\Tilde{\Phi}$~\eqref{eq:changeofbasis}. As can be seen, the function is non-periodic and clearly exhibits the Gibbs phenomenon. The new basis functions $\tilde{\Phi}$ are small in the interior of the interval and grow larger near the boundaries. As explained in~\S\ref{sec:changeofbasis2}, we can obtain information on how to optimally sample for $L^2$-approximation in this non-standard set using the inverse Christoffel function. Fig.\ \ref{fig:FLchristoffel} shows this (near-)optimal sampling distribution, which can be computed numerically after orthonormalizing the approximation set. It clearly shows that many more samples are needed close to the boundary. %Near the boundary, the distribution behaves approximately as $\mathcal{O}(1/x^2)$.
Note that in contrast to other applications of Christoffel theory the numerical orthogonalization is fairly efficient in this case, as the size of the new basis is small.

Incorporating this knowledge into the least squares problem formulation results in a system matrix~\eqref{eq:A} with $M_N = N$ equispaced sample points and $M_K = 2K$ sample points clustered towards the boundary at $x = 0$ and $x = 1$ (characterized by \texttt{[1 ./range(1,1000,K); 1 .- 1 ./range(1,1000,K)]} in Julia notation). The matrix $Z$~\eqref{eq:Z} of the AZ algorithm can then be constructed with $Z_{11}^*$ being the inverse of $A_{11}$. Using the results of Theorem~\ref{thm}, the first AZ equation simplifies to
\[
 (A_{22} - A_{21}Z^*_{11}A_{12}) \mathbf{x}_K \approx \begin{bmatrix}
     -A_{21}Z_{11}^* & I \kern1pt
 \end{bmatrix} \mathbf{b}
\]
where the system matrix is of size $2K \times K$. Assuming $K$ is constant, solving the least squares problem only requires a constant amount of time. However, the cost of the algorithm is still dominated by the construction of the system matrix, requiring $\mathcal{O}(KM\log{M})$ flops.

Fig.\ \ref{fig:FLaccuracy} shows the accuracy of the approximation of $f = e^x + \cos(5(x-0.1)^2)$ \cite[(5.22)]{matthysen2018phd} obtained by solving the reformulated least squares problem compared to solving~\eqref{eq:fourierlegendre}. Both approximations use $K = 5$ and are computed using the AZ algorithm. The reformulated problem is at least as accurate in the $L^2$-norm and approximately a factor of $10$ more accurate pointwise, using only $N+2K$ samples instead of $2N$. For both AZ algorithms $\norm{I-AZ^*}_2 = \mathcal{O}(1)$ holds such that the accuracy of the approximations is very close to the accuracy of the related discrete best approximation, following Lemma~\ref{lm:error}. Importantly, both approximations also have a bounded norm of the coefficient vector $\norm{\mathbf{x}}_2$, such that highly accurate approximations can be recovered despite the ill-conditioning of the system matrices.

\begin{figure}
    \begin{minipage}{0.48\textwidth}
        \centering
        \includegraphics[width=.95\linewidth]{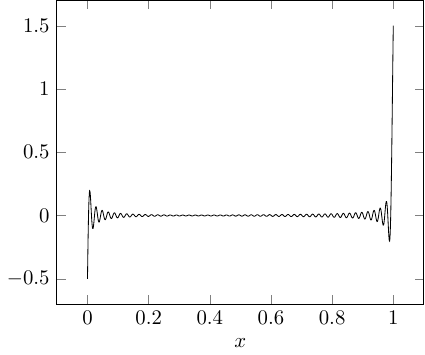}
        \caption{First element of the approximation set $\Tilde{\Phi}$ associated to the first AZ equation for the Fourier + Legendre approximation set.}
        \label{fig:FLbasis}
    \end{minipage}\hfill
    \begin{minipage}{0.48\textwidth}
        \centering
        \includegraphics[width=.95\linewidth]{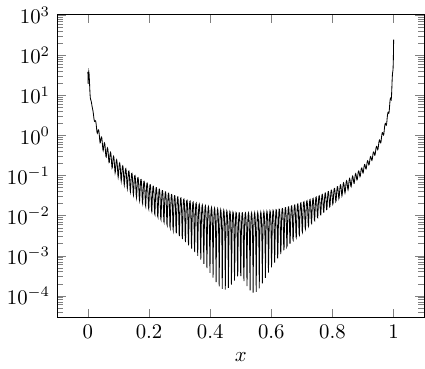}
        \caption{(Near-)optimal sampling distribution for the approximation set $\Tilde{\Phi}$ associated to the first AZ equation for the Fourier + Legendre approximation set.}
        \label{fig:FLchristoffel}
    \end{minipage}
\end{figure}

\begin{figure}
\begin{subfigure}{.49\linewidth}
    \centering
    \includegraphics[width=\linewidth]{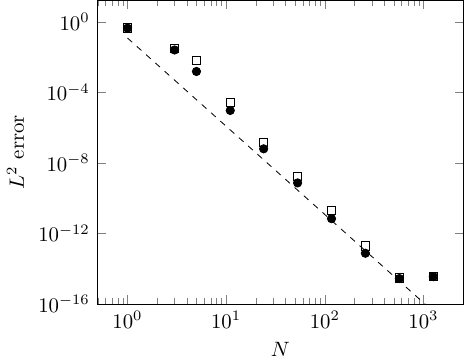}
\end{subfigure}\hfill
\begin{subfigure}{.49\linewidth}
    \centering
    \includegraphics[width=\linewidth]{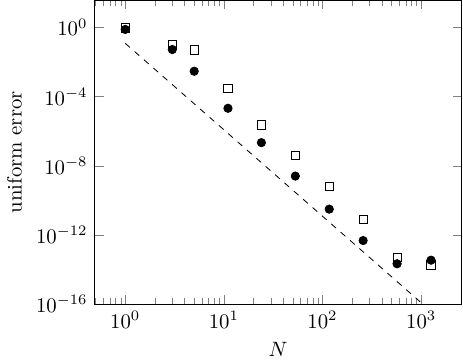}
\end{subfigure}
    \caption{Accuracy of the Fourier + Legendre approximation of $f = e^x + \cos(5(x-0.1)^2)$ with $K = 5$ computed with the AZ algorithm. Squares: discrete approximation on an oversampled equispaced grid ($M_N = 2N$), dots: discrete approximation on an equispaced grid with extra points clustering towards boundaries ($M_N = N, M_K = 2K$). The dashed line marks $\mathcal{O}(N^{-K}) = \mathcal{O}(N^{-5})$.}
    \label{fig:FLaccuracy}
\end{figure}

\subsection{Polynomials augmented with weighted polynomials}\label{sec:example2}
Often a certain characteristic of a function to be approximated is known, such as its oscillatory or singular behaviour. One can then aim at approximating said function in an approximation set which incorporates this characteristic. The set could be of the following form:
\[
\Phi_{N+K} = \{\varphi_n\}_{n=1}^N \cup w(\mathbf{x})\{\psi_k\}_{k=1}^K \qquad \text{with } \psi_k = w(\mathbf{x}) \varphi_k(\mathbf{x})
\]
in which $\varphi_n$ are smooth functions and the weight function $w(\mathbf{x})$ embodies the oscillatory or singular trait of the function $f$. This approximation set is investigated as a frame in \cite[Example 3]{adcock2019frames}. 

\subsubsection{Approximation using an oversampled Chebyshev grid}
As an example, we approximate the Green's function $G(\mathbf{x},\mathbf{y})$ ($\mathbf{x}, \mathbf{y} \in \mathbb{R}^2$) of the the 2D gravity Helmholtz equation~\cite{barnett2015high}. From~\cite[(11)]{barnett2015high} it is known that the function is of the following form:
\[
G(\mathbf{x},\mathbf{y}) = A(\mathbf{x},\mathbf{y})\frac{1}{\log{\lvert \mathbf{x} - \mathbf{y} \rvert}} + B(\mathbf{x},\mathbf{y})
\]
where $A$ and $B$ are analytic in both coordinates of both variables. We will approximate this four-dimensional function for both $\mathbf{x}$ and $\mathbf{y}$ on a semicircle, parametrised by 
\begin{equation}
\label{eq:semicirle}
\vec{\gamma}(s) = \left[ \cos{(2\pi s)}, \sin{(2\pi s)} \right] \text{ for } s \in [0, 0.5].
\end{equation}
The approximation set consists of $N$ bivariate Chebyshev polynomials in the parametrisation variables $s_x$ and $s_y$
\[
\{\varphi_n\}_{n=1}^{\sqrt{N} \times \sqrt{N}} = \{T_{i}(s_x)\kern1pt T_j(s_y)\}_{i,j\kern1pt = \kern1pt (0,0)}^{(\sqrt{N}-1, \kern1pt \sqrt{N}-1)}
\]
and $K$ weighted bivariate Chebyshev polynomials with $w(s_x,s_y) = 1/\log{\lvert \vec{\gamma}(s_x) - \vec{\gamma}(s_y) \rvert }$. The approximation grid is the cartesian product of $2\sqrt{N}$ Chebyshev nodes in the $s_x$-direction and $2\sqrt{N}$ Chebyshev extremae in the $s_y$-direction, such that the function is not evaluated directly at the logarithmic singularity located at $s_x = s_y$. The discrete approximation on the oversampled grid can then again be computed by solving the least squares problem \eqref{eq:fourierlegendre} using the AZ algorithm with $Z_{11}^*$ a left inverse of $A_{11}$, both having fast matrix multiplications using the DCT algorithm. Fig.\ \ref{fig:WPaccuracy} (left) displays the approximation error for $K = 5^2 = 25$. The convergence behaviour is compared to the Chebyshev approximant with $K = 0$. The convergence rate increases significantly by augmenting the approximation set. On Fig.\ \ref{fig:WPaccuracy} (right) the timings of these approximations are displayed, showing that the computational costs of both algorithms differ only by a constant factor.

\subsubsection{Approximation using an oversampled Chebyshev grid with extra points}\label{sec:WCextrapts}

Suppose one wants to increase the accuracy close to the singularity at the diagonal $s_x = s_y$. This can be achieved by adding points close to the diagonal to the oversampled Chebyshev grid, without much affecting the computational cost of the algorithm. For example, this can result in a new least squares problem \eqref{eq:A} with $M_N = 4N$ and $M_K = 2K$. This problem can be solved efficiently using the AZ algorithm with a matrix $Z$ defined by~\eqref{eq:Z}, where $Z_{11}$ is again a left inverse of $A_{11}$. Using the results of Theorem~\ref{thm}, the first AZ equation simplifies to 
\[
\begin{bmatrix}
    A_{12} - A_{11}Z_{11}^*A_{12} \\ A_{22} - A_{21}Z_{11}^*A_{12}
\end{bmatrix} \mathbf{x}_{K} \approx (I-AZ^*) \mathbf{b}.
\]
As explained in~\S\ref{sec:changeofbasis2}, the first AZ equation can be viewed as a discrete approximation problem in a new basis. As opposed to~\S\ref{sec:FLextrapts}, the sample set for this new approximation problem now not only includes the extra $M_K$ points, but also the oversampled grid related to the conventional basis. The extra points can therefore be chosen more freely.

As an example, we again approximate the Green's function, now also adding $2K$ points close to the diagonal at $s_x = s_y$ ($K$ equispaced points at a distance \texttt{1e-3} above the diagonal and $K$ points at the same distance below the diagonal). Fig.\ \ref{fig:WPerrorplot} shows the error plots related to both sampling strategies for $N = 900$ and $K = 25$. Adding the points has a clear effect on the accuracy near the singularity.

\begin{figure}
    \centering
    \includegraphics[width=.5\textwidth]{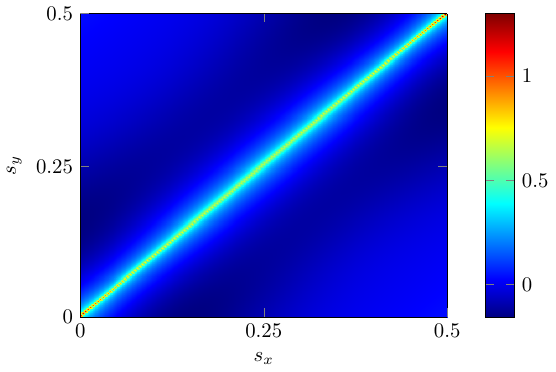}
    \caption{The Green's function $G(\mathbf{x}, \mathbf{y})$ of the 2D gravity Helmholtz equation for $\mathbf{x} = \vec{\gamma}(s_x)$ and $\mathbf{y} = \vec{\gamma}(s_y)$, where $\vec{\gamma}(s)$ parametrises a semicircle \eqref{eq:semicirle}.}
    \label{fig:WC_function}
\end{figure}

\begin{figure}
\begin{subfigure}{.49\linewidth}
    \centering
    \includegraphics[width=\textwidth]{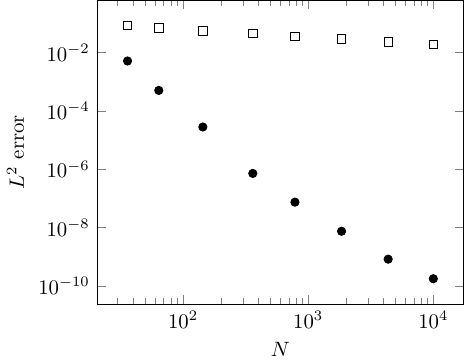}
\end{subfigure}\hfill
\begin{subfigure}{.49\linewidth}
    \centering
    \includegraphics[width=\textwidth]{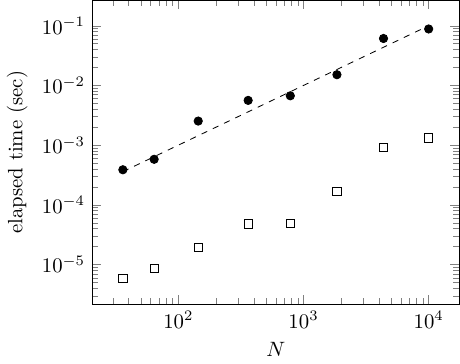}
\end{subfigure}
\caption{The $L^2$ error (left) and the timings (right) for the discrete approximations on an oversampled Chebsyhev grid with $M = 4N$. Dots: Chebsyhev + weighted Chebyshev approximation using $K = 25$, squares: Chebyshev approximation with $K = 0$. The dashed line marks $\mathcal{O}(N)$.}
\label{fig:WPaccuracy}
\end{figure}

\begin{figure}
\begin{subfigure}{.49\linewidth}
    \centering
    \includegraphics[width=\linewidth]{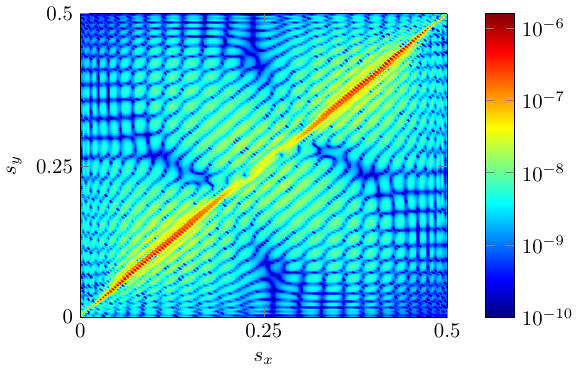}
\end{subfigure}\hfill
\begin{subfigure}{.49\linewidth}
    \centering
    \includegraphics[width=\linewidth]{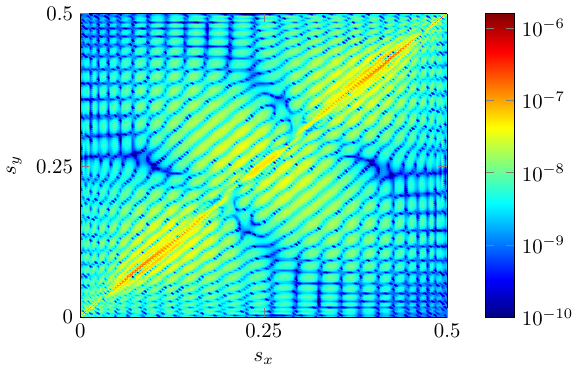}
\end{subfigure}
\caption{Error plot for the Chebyshev + weighted Chebyshev approximations with $N = 900$, $K = 25$, using an oversampled Chebyshev grid $M_N = 4N$. Left: no extra points ($M_K = 0$), right: extra points close to the diagonal ($M_K = 2K$).}
\label{fig:WPerrorplot}
\end{figure}

\section{Application: Enriched Spectral-Galerkin methods} \label{sec:5}
Enriched approximation schemes are often used to solve (partial) differential equations for which the solution exhibits known singular behaviour. As an example, we analyze the applicability of AZ to Enriched Spectral-Galerkin methods to solve elliptic problems. We show that a recent algorithm proposed by Chen and Shen~\cite{chen2018enriched} can itself be interpreted as an AZ algorithm. Pursuing that interpretation further leads to a modified problem formulation and associated AZ algorithm in which some of the assumptions of the existing method can be removed. Crucially, both approaches allow to combine the Galerkin method efficiently with some form of oversampling, such that accurate approximations can still be obtained despite ill-conditioning of the system matrix, as explained in~\S\ref{sec:fna2bound}. 

\subsection{Galerkin method}
We first briefly summarize the problem formulation of~\cite{chen2018enriched}. The goal is to solve an elliptic problem 
\begin{equation} \label{eq:pde}
    \mathcal{L}(u) = f \quad \text{in } \Omega,
\end{equation}
with boundary conditions on $\partial \Omega$, using a weak formulation. Find a function $u$ in a solution space $X$ such that 
\begin{equation*}
    a(u,v) = \langle f, v \rangle, \quad \forall v \in X,
\end{equation*}
where $f$ is a given function in the dual space $X'$ and $a(u,v)$ is a coercive and continuous bilinear form in $X \times X$. Note that the solution space $X$ only includes functions which satisfy the boundary conditions of~\eqref{eq:pde}. A standard Galerkin method aims at approximating the solution $u$ by a function $u_N$ which lies in the span of conventional basis functions $\{\varphi_n\}_{n=1}^N$ such that 
\begin{equation*}
    a(u_N, v_N) = \langle f, v_N \rangle, \quad \forall v_N \in \SPAN(\{\varphi_n\}_{n=1}^N).
\end{equation*}

However, when the first $K$ leading singular terms $\psi_k, k=1\dots K$ of the solution are known, these could be added to the approximation set, resulting in an enriched basis $\Phi_{N+K}$ with the exact form of~\eqref{eq:enrichedbasis}. The straightforward extension of the Galerkin method then amounts to computing $u_N^K \in~\SPAN(\Phi_{N+K})$ such that 
\begin{equation*}
    a(u_N^K, v_N^K) = \langle f, v_N^K \rangle, \quad \forall v_N^K \in \SPAN(\Phi_{N+K}).
\end{equation*}
In order to find $u_N^K$ one needs to solve the following square linear system:
\begin{equation} \label{eq:ESG1}
    \begin{bmatrix}
        A_{11} & A_{12} \\ A_{21} & A_{22} 
    \end{bmatrix} \begin{bmatrix}
        \mathbf{x}_N \\ \mathbf{x}_K
    \end{bmatrix} = \begin{bmatrix}
        \mathbf{f}_N \\ \mathbf{f}_K
    \end{bmatrix}
\end{equation}
where $(A_{11})_{n,n'} = a(\varphi_{n'}, \varphi_n)$, $(A_{12})_{n,k} = a(\psi_k, \varphi_n)$, $(A_{21})_{k,n} = a(\varphi_n, \psi_k)$, $(A_{22})_{k,k'} = a(\psi_{k'}, \psi_k)$, $(\mathbf{f}_N)_n = \langle f, \varphi_n \rangle$ and $(\mathbf{f}_K)_k = \langle f, \psi_k \rangle$.

In~\cite{chen2018enriched}, it is proposed to solve~\eqref{eq:ESG1} using the Schur-complement method which, as mentioned in~\S\ref{sec:3}, is equivalent to the AZ algorithm for enriched bases with $Z$ as defined in~\eqref{eq:Z} having $Z_{11}^* = A_{11}^{-1}$. This method is referred to as \textit{ESG-I}~\cite[\S2.1]{chen2018enriched}. The cost of the algorithm is dominated by $K$ applications of the solver $A_{11}^{-1}$. Although it is efficient, it is mentioned in~\cite{chen2018enriched} that the error can deteriorate significantly due to the ill-conditioning. As explained in~\S\ref{sec:fna2bound}, this is a known phenomenon when approximating in overcomplete sets and the effect can be mitigated by a combination of oversampling and regularization. 

\subsection{Galerkin method combined with smoothness constraints}
The need to incorporate oversampling in the problem formulation was also identified in~\cite{chen2018enriched} and led to a new algorithm referred to as \textit{ESG-II} \cite[\S2.2]{chen2018enriched}. The method relies on the spectral decay of the coefficients $\mathbf{x}_N$ related to the conventional basis functions. In ESG-II, the vanishing of late coefficients is enforced and leads to extra constraints which are used to determine the coefficients in a new basis, consisting of the singular functions minus their approximation in the conventional basis. According to the interpretation of the AZ algorithm described in~\S\ref{sec:changeofbasis2}, we can therefore reformulate ESG-II as an AZ algorithm. Consider the following rectangular system of equations:
\begin{equation} \label{eq:ESG2}
    \begin{bmatrix}
        A_{11} & A_{12} \\ I_{(N-M_K+1:N,\kern1pt:)} & 0
    \end{bmatrix} \begin{bmatrix}
        \mathbf{x}_N \\ \mathbf{x}_K
    \end{bmatrix} = \begin{bmatrix}
        \mathbf{f}_N \\ 0
    \end{bmatrix}
\end{equation}
where $A_{11}, A_{12}$ and $\mathbf{f}_N$ are defined as before and $I_{(N-M_K+1:N,\kern1pt:)}$ contains the last $M_K \geq K$ rows of an $N \times N$ identity matrix. The extra rows impose weakly, in a least squares sense, that the last $M_K$ coefficients in the conventional basis should be small. ESG-II solves this system using the AZ algorithm for enriched bases with $Z$ as defined in~\eqref{eq:Z} having $Z_{11}^* = A_{11}^{-1}$. The cost of the algorithm is dominated by constructing the system matrix $A-AZ^*A$, which requires $K$ applications of the solver $A_{11}^{-1}$. The method has shown to be effective for solving problems with weakly singular solutions~\cite{chen2018enriched} as well as for singularly perturbed problems and singular integral equations~\cite{chen2020high}.

\subsection{Galerkin method combined with collocation}
ESG-II succeeds in producing accurate results for enriched approximation spaces when there is spectral decay of the coefficients in the conventional basis. Yet, this condition can also be avoided by substituting the smoothness constraints for collocation constraints. These constraints impose that the (partial) differential equation is satisfied in a chosen set of collocation points. Again, the constraints can be oversampled leading to a rectangular method. Although this approach loosens the smoothness condition on the conventional basis functions, i.e.\ spectral decay of the coefficients is not expected, it does require a higher order of differentiability of all basis functions, since a strong formulation is used instead of a weak formulation. 

Consider the following rectangular system of equations:
\begin{equation} \label{eq:ESGAZ}
    \begin{bmatrix}
        A_{11} & A_{12} \\ B_{21} & B_{22}
    \end{bmatrix} \begin{bmatrix}
        \mathbf{x}_N \\ \mathbf{x}_K
    \end{bmatrix} = \begin{bmatrix}
        \mathbf{f}_N \\ \mathbf{b}_{M_K}
    \end{bmatrix}
\end{equation}
where $A_{11}, A_{12}$ and $\mathbf{f}_N$ are defined as before and \begin{align*}
(B_{21})_{m,n} = \mathcal{L}(\varphi_n)(t_m) \\
(B_{22})_{m,k} = \mathcal{L}(\psi_k)(t_m) \\
(\mathbf{b}_{M_K})_m = f(t_m)
\end{align*}
define the collocation constraints, where $\{t_m\}_{m=1}^{M_K}$ is a set of $M_K \geq K$ collocation points in $\Omega \setminus \partial\Omega$. This system can again be efficiently solved using the AZ algorithm with $Z$ as defined by~\eqref{eq:Z} having $Z_{11}^* = A_{11}^{-1}$. The cost is again dominated by constructing the system matrix $A-AZ^*A$, which requires $K$ applications of the solver $A_{11}^{-1}$.

\subsection{Example: Poisson equation in a rectangular domain}
As an example, we redo the problem proposed in~\cite[\S3.1]{chen2018enriched}, i.e.\ approximating a weakly singular solution to the 2D Poisson equation in a rectangular domain $\Omega = [-1,1]^2$ with homogeneous Dirichlet boundary conditions. Here, homogenized Jacobi polynomials are used as the spectral basis,
\[
     \{\varphi_n\}_{n=1}^N = \{ \phi_i(x) \phi_j(y): 1 \leq i,j \leq \sqrt{N} \} \quad \text{where } \phi_i(z) = (1-z^2)J_{i-1}^{1,1}(z),
\]
and one leading singular term $\psi_1$ is identified using the results from~\cite{li2005particular} and homogenized to satisfy the boundary conditions. For more details, we refer to~\cite[\S3.1]{chen2018enriched}. On Fig.\ \ref{fig:esg}, the accuracy of ESG-II (using $M_K = 2$) is compared to the Galerkin method combined with collocation constraints (using an $M_K = 5^2$ equispaced collocation grid). As a reference solution, ESG-II with $N = 900$ is used. Both methods converge similarly. The error of the standard Galerkin method without enrichment is shown for comparison. As expected, its accuracy is very poor.

A direction for future work is to study the real-life computational cost of the algorithm for optimized implementations. Moreover, the results strongly motivate further research on the applicability of the AZ algorithm for other types of enriched solvers for PDEs, such as FEM solvers.

\begin{figure}
    \centering
    \includegraphics[width=.6\textwidth]{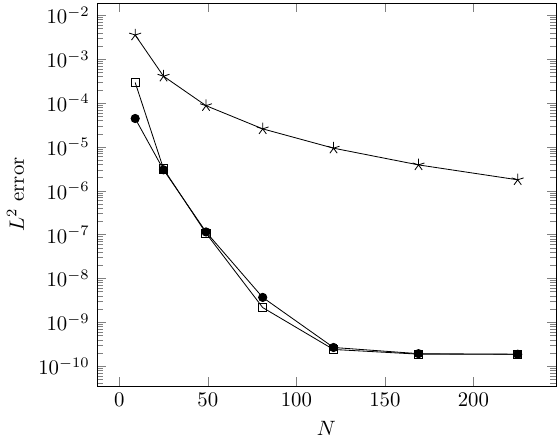}
    \caption{Accuracy of the solution to the 2D Poisson equation in a rectangular domain. Stars: standard Galerkin method without enrichment (i.e.\ $K = 0$), squares: Galerkin method combined with smoothness constraints (ESG-II) using $K = 1$ and $M_K = 2$, dots: Galerkin method combined with collocation using $K = 1$ and $M_K = 25$. The results for the standard Galerkin method and ESG-II also follow from~\cite[Fig.\ 4]{chen2018enriched}.}
    \label{fig:esg}
\end{figure}

\bibliographystyle{abbrv}
\bibliography{references}

\end{document}